\documentclass{article}
\usepackage{graphicx} 
\usepackage[utf8]{inputenc}
\usepackage{geometry}
\usepackage{amsmath,amsthm,verbatim,amssymb,amscd, graphicx}
\usepackage{natbib}

\usepackage{bbm}
\usepackage{verbatim}
\usepackage{xcolor}
\usepackage{hyperref}
\usepackage{amssymb,bm}
\usepackage{mathrsfs} 
\usepackage{authblk}

\usepackage{graphicx}
\usepackage{subcaption}
\newtheorem{theorem}{Theorem}[section]

\newtheorem{remark}{Remark}[section]
\newtheorem{lemma}{Lemma}[section]

\renewcommand{\leq}{\leqslant}
\renewcommand{\geq}{\geqslant}
\renewcommand{\le}{\leqslant}
\renewcommand{\ge}{\geqslant}
\renewcommand{\hat}{\widehat}

\newcommand{\argmin}{\mathop{\mathrm{arg}\,\mathrm{min}}}

\usepackage{xspace}

\newcommand{\iid}{i.i.d.\@\xspace}

\theoremstyle{definition}
\newtheorem{definition}{Definition}

\title{Gradient-free stochastic optimization of derivatives under strong convexity}

\author[1]{Arya Akhavan}
\author[2]{Sirine Louati}
\author[2]{Alexandre B. Tsybakov}
\affil[1]{University of Oxford}
\affil[2]{CREST, ENSAE, IP Paris}

\date{}

\begin{document}

\maketitle
\begin{abstract}
We consider the problem of minimizing the $k$-th order partial 
derivative $f=\partial_j^k g$ of an unknown function $g$ along a 
fixed coordinate direction $j$, based on noisy  
queries of $g$. Assuming that $g$ has H\"older regularity
${\beta+k}$ for some $\beta\ge 2$, that $f$ is 
strongly convex on a compact convex set $\Theta\subset\mathbb{R}^d$ 
and that $g$ and $f$ satisfy mild boundedness and Lipschitz 
regularity conditions on $\Theta$, we propose a kernel-based 
estimator of $\nabla f$ and analyze the projected stochastic 
gradient algorithm driven by this estimator. We obtain a 
non-asymptotic upper bound on the optimization error of the order 
$d^{(2\beta+k-1)/(\beta+k)}\,N^{-(\beta-1)/(\beta+k)}$, where 
$N$ is the total number of queries. We also
establish a minimax lower bound of the order 
$N^{-(\beta-1)/(\beta+k)}$ showing that this rate is optimal 
in $N$ over all sequential algorithms. 
\end{abstract}

\section{Introduction}
\label{sec:intro}

Gradient-free (or zero-order) optimization is widely used in statistics, machine learning, and scientific computing, particularly in settings where the objective function can only be accessed through noisy queries \citep{conn2009introduction,spall2002multivariate}. Such situations arise, for instance, when running complex simulations, conducting physical 
experiments or querying expensive black-box models. In these contexts, gradient information is not directly available and optimization procedures typically rely on randomized perturbations and smoothing techniques to extract local information from function values. 
These problems are also closely related to bandit and sequential decision-making frameworks, where optimization must be performed under partial feedback \citep{flaxman2004online,hazan2014bandit,shamir2017optimal}. While minimizing a function from its own noisy queries is now well understood, 
the present paper addresses a harder problem: minimizing the $k$-th derivative of an unknown function observed only through noisy zero-order queries of the function.

The literature on zero-order and bandit optimization spans several research communities, including stochastic approximation, online learning and 
nonparametric statistics. A broad range of algorithmic and information-theoretic results has established convergence guarantees for optimizing unknown functions from noisy feedback using 
randomized gradient surrogates, smoothing schemes or finite-difference constructions 
\citep{kiefer1952stochastic, polyak1990optimal, dippon2003accelerated, agarwal2010optimal,ghadimi2013stochastic,shamir2013complexity, duchi2015optimal,
bach2016highly,nesterov2017random,shamir2017optimal,locatelli2018adaptivity,akhavan2020exploiting,akhavan2021distributed,akhavan2024estimating,akhavan2024gradient}. The optimal behavior of cumulative regret and optimization error is now well understood under different feedback structures, noise models and regularity assumptions. By contrast, to the best of our knowledge, optimizing functions from indirect observations was not considered, despite its relevance in a variety of statistical 
and computational contexts.


Problems involving indirect observations have long been studied in nonparametric statistics, particularly in the context of density deconvolution and errors-in-variables models, starting from the works \citep{carroll1988optimal,devroye1989consistent,stefanski1990rates}. Optimal convergence rates for recovering probability densities in the problem of deconvolution have been established in various settings \citep{butucea2008sharp,Lepski2019OracleIA}. For a book-length account on density deconvolution see \citep{Meister2009DeconvolutionPI}. In the optimization perspective, the major problem related to density estimation is that of estimating the mode, that is, the maximizer of a probability density. While the optimal rates for this problem in direct observation setting are known since \cite{tsybakov1990recursive} and depend exponentially on the dimension, the case of indirect observations was treated only for one-dimensional deconvolution model \citep{wieczorek2010optimal}. In particular, the results in \citep{wieczorek2010optimal} establish optimal rates for estimating the maximizer of the $k$-th derivative of a probability density $p$ based on an i.i.d. sample from~$p$. The problem of stochastic optimization that we are considering is different since the query points are not i.i.d. and can be chosen sequentially by the learner. Thanks to this possibility of choice, as shown below, one can construct algorithms with rates of convergence that depend on the dimension~$d$ only through a weak factor not exceeding 
$d^2$. 

Minimizing the derivative $g^{(k)}$ using noisy queries of $g$ can also be related to the literature on linear inverse problems in nonparametric regression and Gaussian white noise model, where rate optimal estimators are obtained in various settings (see, e.g., \cite{bissantz2008statistical,cavalier2008nonparametric}). Estimating the $k$-th derivative is a basic special instance of linear inverse problem. That line of work proposes tools for estimation of $g^{(k)}$ as a whole object under the $L_q$ loss rather than estimation of its minimizer. Notably, that literature deals with model that do not allow for sequential choice of queries, so that the minimax optimal rates depend exponentially on the dimension.


In this paper we study the problem of minimizing the $k$-th order
partial derivative $f=\partial_j^k g$ of an unknown function
$g$ over a compact convex set $\Theta\subset\mathbb{R}^d$, given a
budget of $N$ noisy zero-order queries of $g$ at sequentially chosen
points. Here $j\in\{1,\dots,d\}$ is a fixed coordinate and
$\partial_j^k$ denotes the $k$-th order partial derivative along the
$j$-th coordinate so that $f$ corresponds to the multi-index
$s=(0,\dots,0,k,0,\dots,0)$ with a single nonzero entry equal to $k$
in position $j$.
Assuming that $g$ has H\"older regularity $\beta+k$ for some 
$\beta\ge 2$ and that $f$ is strongly convex on $\Theta$, we 
establish non-asymptotic minimax upper and lower bounds on the 
optimization error that match in $N$. The upper bound scales as 
$d^{(2\beta+k-1)/(\beta+k)}\,N^{-(\beta-1)/(\beta+k)}$, and the 
matching lower bound is of the order 
$N^{-(\beta-1)/(\beta+k)}$. The gap between the upper 
and lower bounds is represented by a mild dimension dependent factor. The question of improving this factor remains open.


\paragraph{Notation.}
Throughout the paper, $\langle\cdot,\cdot\rangle$ and
$\|\cdot\|$ denote the standard inner product and Euclidean norm
on $\mathbb{R}^d$. We denote by $\Theta\subset\mathbb{R}^d$ a
compact convex set with non-empty interior, by $\Pi_\Theta$ the
Euclidean projection onto $\Theta$, and by
$R:=\sup_{x,y\in\Theta}\|x-y\|$ its diameter. We write
\[
\Theta^+:=\{x\in\mathbb{R}^d:\mathrm{dist}(x,\Theta)\leq 1\}
\]
for the unit enlargement of $\Theta$, which is compact since $\Theta$
is compact. Here $\mathrm{dist}(x,\Theta)$ is the Euclidean distance from $x$ to $\Theta$. We fix once and for all an integer
\(k\in\mathbb N_0:=\{0,1,2,\ldots\}\) and a coordinate direction
\(j\in\{1,\ldots,d\}\). We denote by $e_j$ the $j$-th canonical basis vector in $\mathbb{R}^d$. We use the convention
\(
\partial_j^0 g := g .
\)
Thus, whenever the target function is defined as \(f=\partial_j^k g\),
the case \(k=0\) corresponds to the direct problem \(f=g\). In this
case the coordinate \(j\) is immaterial for the definition of \(f\),
but we keep it in the notation because the estimator introduced below
uses an auxiliary perturbation in the \(e_j\)-direction. For \(k\ge1\),
\(f=\partial_j^k g\) is the \(k\)-th  partial derivative
of \(g\) over the coordinate direction $j$.

For $L>0$ and $\beta>0$, the H\"older class
$\mathcal{F}_\beta(L)$ is defined in Section~\ref{sec:prel}. A
continuously differentiable function $f:\mathbb{R}^d\to\mathbb{R}$ is said to be
$\alpha$-strongly convex on a convex set $S\subset\mathbb{R}^d$
if
\[
f(y)\geq f(x)+\langle\nabla f(x),y-x\rangle
+\tfrac{\alpha}{2}\|y-x\|^2,\qquad\forall x,y\in S.
\]

For parameters $L,L_f,G,G_g,\bar{L}>0$, $\alpha>0$ and
$\beta>1$, we define the class
$\mathcal{F}'_{\alpha,\beta,k}(L,L_f,G,G_g,\bar{L})$ of all
functions $g:\mathbb{R}^d\to\mathbb{R}$ such that, with
$f:=\partial_j^k g$:
\begin{enumerate}
\item[(a)] $g\in\mathcal{F}_{\beta+k}(L)$;
\item[(b)] $f$ is $\alpha$-strongly convex on $\Theta$;
\item[(c)] $\nabla f$ is $L_f$-Lipschitz on $\Theta$ with respect to the Euclidean norm;
\item[(d)] $\|\nabla f(x)\|\leq G$ for all $x\in\Theta$;
\item[(e)] $\|\nabla g(x)\|\leq G_g$ for all $x\in\Theta^+$;
\item[(f)] $\nabla g$ is $\bar{L}$-Lipschitz on $\Theta^+$ with respect to the Euclidean norm;
\item[(g)] $f$ attains its minimum on $\Theta$ at a point
$x^\star\in\mathrm{int}(\Theta)$.
\end{enumerate}
The minimizer $x^\star$ in (g) is unique by (b).

\paragraph{Contributions:}
\begin{itemize}
    \item \textbf{Bias--variance control for derivative estimators.} In Section~\ref{sec:main}, we analyze a kernel based estimator of the gradient $\nabla f(x)$ 
    constructed from 
    paired function
    queries. Lemma~\ref{lem:bias-multi} provides a non-asymptotic bound on the smoothing bias.
    Lemma~\ref{lem:var-multi-second-moment}
    establishes a second moment bound that captures the joint effects of noise variance, smoothing
    and the scale of $\nabla f$. Together, these results characterize the fundamental bias--variance trade-off governing derivative estimation from zero-order data.

   \item \textbf{Finite-sample convergence rates for gradient-free 
optimization.} In Section~\ref{sec:main}, we study a projected stochastic 
gradient algorithm driven by the proposed kernel-based gradient 
approximation. Theorem~\ref{thm:upper-bound} establishes the following 
non-asymptotic upper bound on the expected optimization error. For a 
suitable choice of parameters of the algorithm, the algorithm returns an estimator $x_T$ such that after $T$ steps with $2d$ queries per step,
\[
\sup_{g \in \mathcal{F}'_{\alpha,\beta,k}(L,L_f,G,G_g,\bar{L})}
\mathbb{E}\bigl[f(x_T)-\min_{x\in \Theta}f(x)\bigr]
\;\le\; C\,d^{(2\beta+k-1)/(\beta+k)}\,N^{-(\beta-1)/(\beta+k)},
\]
where $N = 2dT$ is the total number of queries, and $C > 0$ is a 
constant independent of $N$ and $d$.

   \item \textbf{Matching minimax lower bound.}
We complement the above upper bound with a minimax lower bound 
establishing that the rate $N^{-(\beta-1)/(\beta+k)}$ is optimal for 
$\beta \geq 2$. Theorem~\ref{thm:lower-bound-bumps} proves that
\[
\inf_{\mathrm{alg}\;\Phi}\sup_{g \in \mathcal{F}'_{\alpha,\beta,k}(L,L_f,G,G_g,\bar{L})}
\mathbb{E}\bigl[f(\hat{x}_N(\Phi))-\min_{x\in \Theta}f(x)\bigr]
\;\ge\; c\,N^{-(\beta-1)/(\beta+k)},
\]
where the infimum is over all sequential algorithms $\Phi$ of choosing $N$ query points and all estimators $\hat{x}_N=\hat{x}_N(\Phi)$ based on these queries, and $c > 0$ is a constant independent of  
$N$ and~$d$. 
\end{itemize}

This paper is organized as follows. Section~\ref{sec:problem} introduces the statement of the problem. Section~\ref{sec:related} discusses some related work. Section~\ref{sec:prel} presents the
smoothness assumptions and the proposed kernel-based estimator. Section~\ref{sec:main} contains the main results.  The proofs are deferred to Section~\ref{sec:proofs}.

\section{Problem setting}
\label{sec:problem}

Let $g:\mathbb{R}^d\to\mathbb{R}$ be an unknown function 
belonging to the class 
$\mathcal{F}'_{\alpha,\beta,k}(L,L_f,G,G_g,\bar{L})$ defined in 
the introduction, for some $\alpha,L,L_f,G,G_g,\bar{L}>0$, 
$\beta>1$, and $k\geq 0$. We consider the target function  
$f:=\partial_j^k g$ that 
is the $k$-th order partial derivative of $g$ 
along the coordinate direction $j\in\{1,\dots,d\}$. We 
study the problem of minimizing $f$ over the compact convex 
set $\Theta$, that is, of approximating
\[
x^\star:=\argmin_{x\in\Theta}f(x),
\qquad
f^\star:=f(x^\star),
\]
based on noisy zero-order evaluations of $g$ at query points 
that the learner chooses sequentially.

\medskip
\noindent\textit{Sequential oracle.}
The data is generated by a sequential interaction between the
learner and a stochastic oracle. At each instance $t=1,2,\ldots,N$,
the learner selects a query point $z_t\in\mathbb{R}^d$ and
observes
\begin{equation}\label{observations}
 y_t = g(z_t)+\xi_t,   
\end{equation}
where $\xi_t$ is a random variable with $\mathbb{E}[\xi_t^2]\leq\sigma^2$.

\begin{definition}[Sequential algorithm]
\label{def:seq-alg}
A \emph{sequential algorithm} is any procedure for choosing the query
points $z_1,\dots,z_N\in\mathbb{R}^d$ such that, for every
$t\in\{1,\dots,N\}$,
\[
z_t=F_t\bigl((z_i,y_i)_{i=1}^{t-1},\,\zeta_t\bigr),
\]
where $F_t$ is a measurable function and $\zeta_t$ is a randomization
variable chosen by the learner, independent of the past observations
$(z_i,y_i)_{i=1}^{t-1}$. The randomization variables $(\zeta_t)_t$ are
independent of the noise sequence $(\xi_t)_t$. 
\end{definition}

We refer to $N$ as the \emph{oracle budget}. When the algorithm we
analyze in Section~\ref{sec:prel} structures the queries into $T$
steps of $2d$ queries (so that $N=2dT$), we will write
$\hat{x}_N=x_T$, where $x_T$ denotes the iterate produced after the
$T$-th step.

\medskip
\noindent\textit{Performance criterion.}
The performance of an estimator $\hat{x}_N$ is 
measured by the optimization error
\[
\mathbb{E}\bigl[f(\hat{x}_N)-f^\star\bigr].
\]
We find the minimax optimal rate of decay of this quantity 
on the class $\mathcal{F}'_{\alpha,\beta,k}(L,L_f,G,G_g,\bar{L})$. The upper and lower bounds that we 
establish in Theorems~\ref{thm:upper-bound} 
and~\ref{thm:lower-bound-bumps}, respectively, identify the optimal rate as a function of $N$.

\medskip
\noindent\textit{Indirect observations.}
Our analysis covers all $k\geq 0$. The novel regime is $k\geq 1$, which
differs fundamentally from the standard zero-order optimization problem
corresponding to $k=0$, for which the minimax rates are studied 
in~\cite{polyak1990optimal,shamir2013complexity,akhavan2020exploiting,
akhavan2024gradient}. For $k\geq 1$, the learner does not observe noisy
values of the target function $f$ itself but only noisy values of $g$, of which $f$ is a $k$-th order derivative. Recovering $f$ therefore requires 
extracting derivative information from zero-order data. This is an 
indirect optimization problem in the spirit of inverse problems 
in nonparametric statistics.  Establishing matching lower bounds requires 
constructing families of functions $g$ that are nearly 
indistinguishable from noisy queries of $g$ alone, yet 
induce well-separated minimizers for their $k$-th derivatives. 
We carry out such a construction in 
Theorem~\ref{thm:lower-bound-bumps}.


\section{Related work}
\label{sec:related}

A major part of related work deals with zero-order optimization, where the goal is to minimize an unknown function using only noisy function queries updated in a sequential manner. This literature originates from stochastic approximation methods such as the Kiefer--Wolfowitz procedure 
\citep{kiefer1952stochastic} and random perturbation schemes mentioned by \cite{nemirovski1983mirror} and
\citep{spall2002multivariate} among others, and has developed into a rich theory encompassing complexity bounds and optimal algorithms under different observation scenarios. Representative results include 
convergence guarantees under bandit and stochastic feedback, see 
\citep{flaxman2004online,agarwal2010optimal,jamieson2012,ghadimi2013stochastic,
duchi2015optimal,shamir2017optimal,balasubramanian2021zeroth,nesterov2017random} and the references cited therein. This work primarily focuses on the settings where the target function and/or its gradient are Lipschitz continuous and additionally the target function is convex or strongly convex. 

A related line of research investigates how higher order smoothness of the target function can improve gradient estimation and optimization error in zero-order and bandit settings. Using smoothing and randomization techniques, several works show that additional regularity can reduce estimator bias and accelerate 
optimization \citep{polyak1990optimal,bach2016highly,akhavan2020exploiting,akhavan2021distributed,novitskii2021improved,akhavan2024gradient,yu2024stochastic,Akhavan2025GradientfreeSO}. These 
contributions deal with the case $k=0$ and do not address the statistical complexity of optimizing higher order derivatives from zero-order data.

Our work is also related to the literature on optimization of functions in nonparametric regression and density estimation settings. There, the main difference is that the observations are i.i.d. rather than sequentially chosen, so that one deals with a passive rather than active scheme of observation. Consequently, the best rates of estimating the mode of the probability density \cite{tsybakov1990recursive} and the minimizer of the nonparametric regression function in the passive scheme (\cite{tsybakov1990passive}, see also \cite{nazin-polyak-tsybakov-1989,nazin-polyak-tsybakov,passive2022,akhavan2024estimating}) are substantially slower than in the active (sequential) scheme that we consider here. They mimic the classical nonparametric estimation rates and depend exponentially on the dimension. Again, this literature does not cover optimization of the derivative of order $k$. The only exception is the paper \cite{wieczorek2010optimal} establishing the minimax rate for estimation of the mode of probability density in one-dimensional deconvolution problem.

To the best of our knowledge, no prior work provides a minimax analysis for optimizing higher order derivatives using noisy zero-order queries under a sequential observation scheme. The present paper fills this gap by establishing matching upper and lower bounds showing how smoothness, dimension, and the order of derivative jointly define the 
fundamental statistical limits of the problem, with the optimal rate $N^{-(\beta-1)/(\beta+k)}$. 

\section{Preliminaries}
\label{sec:prel}

This section introduces the H\"older class, the oracle model 
implementing the gradient estimator and the kernel-based 
construction used throughout the analysis.

\subsection{H\"older class}

For $\beta>0$ we let $\ell=\lfloor\beta\rfloor$ denote the largest integer 
strictly less than $\beta$. 
We define the H\"older class 
$\mathcal{F}_\beta(L)$ as the set of all functions 
$\phi:\mathbb{R}^d\to\mathbb{R}$ that are $\ell$ times 
continuously differentiable on $\mathbb{R}^d$ and satisfy
\begin{equation}
\label{eq:holder}
\biggl|\phi(z)-\sum_{|m|\leq\ell}\frac{1}{m!}D^m\phi(x)(z-x)^m
\biggr|\leq L\|z-x\|^\beta,
\qquad\forall\,x,z\in\mathbb{R}^d,
\end{equation}
where $m=(m_1,\dots,m_d)\in\mathbb{N}^d$ is a multi-index, 
$|m|=\sum_{i=1}^d m_i$, $m!=m_1!\cdots m_d!$ and 
$D^m\phi=\partial_1^{m_1}\cdots\partial_d^{m_d}\phi$. 

\medskip

The smoothness assumption underlying our analysis is that 
$g\in\mathcal{F}'_{\alpha,\beta,k}(L,L_f,G,G_g,\bar{L})$ for some 
$\beta>1$ and $k\geq 0$. The condition $\beta>1$ entails 
$\lfloor\beta+k\rfloor\geq k+1$ so that $g$ admits at least 
$k+1$ continuous derivatives on $\mathbb{R}^d$. By Schwarz's 
theorem, mixed partial derivatives of order $k+1$ of $g$ commute 
and therefore the partial derivatives of $f=\partial_j^k g$ 
satisfy
\begin{equation}
\label{eq:schwarz}
\partial_l f(x)=\partial_l\partial_j^k g(x)
=\partial_j^k\partial_l g(x),\qquad l=1,\dots,d.
\end{equation}

By condition (g) of the class, the minimizer 
$x^\star\in\mathrm{int}(\Theta)$ satisfies the first-order 
optimality condition $\nabla f(x^\star)=0$, which combined with 
the $L_f$-Lipschitz continuity of $\nabla f$ on $\Theta$ yields 
the standard smoothness inequality
\begin{equation}
\label{eq:smoothness}
f(x)-f(x^\star)\leq\frac{L_f}{2}\|x-x^\star\|^2,
\qquad\forall\,x\in\Theta.
\end{equation}

\subsection{Oracle model and gradient estimator}

The algorithm we analyze structures the $N$ queries into 
$T$ steps of $2d$ queries with one pair per 
coordinate direction. Thus, we consider $N=2dT$. At step 
$t\in\{0,\dots,T-1\}$, the current iterate $x_t\in\Theta$ is obtained from the projected stochastic 
gradient recursion given by
\begin{equation}
\label{eq:sgd}
x_{t+1}:=\Pi_\Theta\bigl(x_t-\eta_t\,\tilde{g}_t\bigr),
\qquad t=0,1,\dots,T-1,
\end{equation}
initialized at an arbitrary $x_0\in\Theta$, where $\tilde{g}_t$ is a gradient estimator and $\eta_t >0$ is a step size. We define the query points and $\tilde{g}_t$ in the following manner. For a smoothing parameter $h>0$, for each coordinate 
$l\in\{1,\dots,d\}$, the algorithm queries the oracle at the two 
points 
$x_t+hU_{t,l}e_j\pm hV_{t,l}e_l\in\mathbb{R}^d$ and observes
\begin{equation}
\label{eq:oracle}
y^+_{t,l}=g(x_t+hU_{t,l}e_j+hV_{t,l}e_l)+\varepsilon^+_{t,l},
\qquad
y^-_{t,l}=g(x_t+hU_{t,l}e_j-hV_{t,l}e_l)+\varepsilon^-_{t,l}.
\end{equation}
The randomization variables $U_{t,l},V_{t,l}$, chosen by the learner, are \iid uniform 
on $[-1/2,1/2]$ and the noise variables 
$\varepsilon^\pm_{t,l}$ are such that
$\mathbb{E}[(\varepsilon^\pm_{t,l})^2]\leq\sigma^2$. The set of randomization all variables $(U_{t,l},V_{t,l})$ is independent of the set of all noise variables $(\varepsilon^\pm_{t,l})$. Note that no independence or zero-mean assumption is imposed on the noises $\varepsilon^\pm_{t,l}$. 

\medskip

Set $\ell_*:=\lfloor\beta+k\rfloor=k+\lfloor\beta\rfloor$. The 
gradient estimator relies on two bounded kernels supported on 
$[-1/2,1/2]$. The first kernel $K_k:[-1/2,1/2]\to\mathbb{R}$, is 
associated with the derivative direction $j$ and is required to 
satisfy the moment conditions
\begin{equation}
\label{eq:Kk-moments}
\int u^a K_k(u)\,du=0\quad\text{for }a\in\{0,\dots,\ell_*\}\setminus\{k\},
\qquad
\int u^k K_k(u)\,du=k!.
\end{equation}
The second, $K_1:[-1/2,1/2]\to\mathbb{R}$, is associated with 
the gradient direction $l$ and is required to satisfy
\begin{equation}
\label{eq:K1-moments}
\int v^b K_1(v)\,dv=0\quad\text{for } b=0 \text{ and for odd } b\in\{3,5,\dots,\ell_*\},
\qquad
\int v\,K_1(v)\,dv=1.
\end{equation}
Bounded kernels satisfying these conditions can be constructed 
explicitly as weighted sums of Legendre polynomials; 
see~\citep[Section~1.2.2]{tsybakov2009nonparametric} or~\cite{bach2016highly}. We denote 
by $\|K_k\|_\infty,\|K_1\|_\infty<\infty$ their sup-norms and we 
note that $\mathbb{E}[K_k(U)^2]$ and $\mathbb{E}[K_1(V)^2]$ are 
finite constants depending only on $k$ and $\ell_*$.

\medskip

The kernel-based estimator of the $l$-th component of 
$\nabla f(x_t)$ is then defined by
\begin{equation}
\label{eq:grad-estimator}
\tilde{g}_{t,l}
:=\frac{K_k(U_{t,l})\,K_1(V_{t,l})}{2h^{k+1}}
\bigl(y^+_{t,l}-y^-_{t,l}\bigr),\qquad l=1,\dots,d.
\end{equation}
The perturbation along $e_j$ paired with the kernel $K_k$ 
extracts the $k$-th order derivative in direction $j$, while the 
perturbation along $e_l$ paired with $K_1$ extracts the 
first-order derivative in direction $l$. Together, they recover 
$\partial_l\partial_j^k g(x_t)=\partial_l f(x_t)$ up to a bias of 
order $h^{\beta-1}$, as established in 
Lemma~\ref{lem:bias-multi}. We denote by 
$\tilde{g}_t:=(\tilde{g}_{t,1},\dots,\tilde{g}_{t,d})\in
\mathbb{R}^d$ the resulting full gradient estimator.



\section{Main results}
\label{sec:main}

This section presents our two main results which are a non-asymptotic upper 
bound on the optimization error of the projected stochastic gradient 
algorithm driven by the kernel-based estimator and a matching 
minimax lower bound. The proofs are deferred to 
Section~\ref{sec:proofs}.

\subsection{Upper bound}

The analysis of the projected stochastic gradient algorithm 
relies on a sharp control of the bias and the second moment of 
the gradient estimator $\tilde{g}_t$. The next two lemmas 
establish such bounds.

\begin{lemma}[Bias of the gradient estimator]
\label{lem:bias-multi}
Let $g\in\mathcal{F}_{\beta+k}(L)$ with $\beta>1$ and $k\geq 0$. 
For every $x_t\in\mathbb{R}^d$ and every $l\in\{1,\dots,d\}$,
\[
\bigl|\mathbb{E}[\tilde{g}_{t,l}\mid x_t]-\partial_l f(x_t)\bigr|
\leq C_{\mathrm{bias}}\,h^{\beta-1},
\]
where 
\[
C_{\mathrm{bias}}
:=L\,\bigl(\mathbb{E}[K_k(U)^2]\bigr)^{1/2}
\bigl(\mathbb{E}[K_1(V)^2]\bigr)^{1/2}.
\]
Consequently,
\[
\bigl\|\mathbb{E}[\tilde{g}_t\mid x_t]-\nabla f(x_t)\bigr\|^2
\leq d\,C_{\mathrm{bias}}^2\,h^{2(\beta-1)}.
\]
\end{lemma}

\begin{lemma}[Second moment of the gradient estimator]
\label{lem:var-multi-second-moment}
Let $g\in\mathcal{F}'_{\alpha,\beta,k}(L,L_f,G,G_g,\bar{L})$. For 
any $x_t\in\Theta$ and $h\in(0,1]$,
\[
\mathbb{E}\bigl[\|\tilde{g}_t\|^2\mid x_t\bigr]
\leq 4G^2+4d\,C_{\mathrm{bias}}^2\,h^{2(\beta-1)}
+2d\,C_{\mathrm{var}}\,h^{-2(k+1)},
\]
where 
\[
C_{\mathrm{var}}
:=\frac{\mathbb{E}[K_k(U)^2]\,\mathbb{E}[K_1(V)^2]}{4}
\Bigl(\frac{3\bar{L}^2}{16}+6G_g^2+8\sigma^2\Bigr).
\]
\end{lemma}

The proofs of Lemmas~\ref{lem:bias-multi} 
and~\ref{lem:var-multi-second-moment} are given in 
Section~\ref{sec:proofs}. The bias bound is governed by the 
H\"older regularity $\beta$ of $f$ and exploits the moment 
cancellation properties of the 
kernels~\eqref{eq:Kk-moments}--\eqref{eq:K1-moments}. The variance 
bound captures the joint contributions of the noise variance 
$\sigma^2$, the H\"older constant $\bar{L}$ of $\nabla g$ and the 
gradient bound $G_g$, all amplified by the factor $h^{-2(k+1)}$ 
arising from the $h^{-(k+1)}$ scaling of the estimator. 
Combining these two bounds with a \cite{chung1954stochastic} type recursion 
yields the following non-asymptotic convergence rate.

\begin{theorem}[Upper bound]
\label{thm:upper-bound}
Let $\beta\ge 2$, $k\geq 0$, and let $\alpha,L,L_f,G,G_g,\bar{L}>0$ 
be given constants. Consider the projected stochastic gradient 
algorithm~\eqref{eq:sgd} with step size $\eta_t=\gamma/(t+t_0)$, 
where $\gamma\geq 4/\alpha$ and $t_0\geq\alpha\gamma$ and smoothing
parameter $h=\min\bigl(\kappa\,T^{-1/(2(\beta+k))},1\bigr)$ for 
some $\kappa>0$. Set $N=2dT$. Then, for any initialization 
$x_0\in\Theta$,
\[
\sup_{g\in\mathcal{F}'_{\alpha,\beta,k}(L,L_f,G,G_g,\bar{L})}
\mathbb{E}\bigl[f(x_T)-f(x^\star)\bigr]
\leq C\,d^{(2\beta+k-1)/(\beta+k)}\,N^{-(\beta-1)/(\beta+k)},
\]
where $C>0$ is a constant depending only on 
$\alpha,\gamma,t_0,\kappa,L_f,G,R,C_{\mathrm{bias}},C_{\mathrm{var}}$ 
and independent of $N$ and $d$. 
\end{theorem}

\begin{remark}[Bias-variance trade-off]
\label{rem:bias-variance}
The proof of Theorem~\ref{thm:upper-bound} shows that, for any 
$h\in(0,1]$,
\[
\mathbb{E}[f(x_T)-f(x^\star)]
\leq\underbrace{\frac{C_6}{T}}_{\text{initialization}}
+\underbrace{\frac{C_7\,d\,h^{-2(k+1)}}{T}}_{\text{variance}}
+\underbrace{C_8\,d\,h^{2(\beta-1)}}_{\text{bias}},
\]
where $C_6,C_7,C_8>0$ depend only on the model parameters. This 
bound exhibits a bias-variance trade-off in $h$. The 
variance term scales as $d\,h^{-2(k+1)}/T$ and blows up as 
$h\to 0$ due to the factor $h^{-(k+1)}$ in the 
estimator~\eqref{eq:grad-estimator}. The bias term scales as 
$d\,h^{2(\beta-1)}$ and vanishes as $h\to 0$, its rate being 
governed by the H\"older regularity $\beta$ of $f$. The optimal 
trade-off $h\asymp T^{-1/(2(\beta+k))}$ makes both terms of order 
$d\,T^{-(\beta-1)/(\beta+k)}$.




\end{remark}

\subsection{Lower bound}

We now show that the rate $N^{-(\beta-1)/(\beta+k)}$ obtained in 
Theorem~\ref{thm:upper-bound} cannot be improved by any sequential 
algorithm operating under the zero-order oracle model. The proof relies on a reduction to the problem of distinguishing between
two hypotheses, in which we construct a pair of functions $g_+,g_-$ that are nearly indistinguishable from 
noisy observations of $g$ alone, yet induce well-separated 
minimizers for their $k$-th derivatives 
$f_\pm=\partial_j^k g_\pm$. 

\begin{theorem}[Lower bound]
\label{thm:lower-bound-bumps}
Let $\beta\geq 2$, $k\geq 0$, and assume the
oracle model \eqref{observations}, where $\xi_t$'s are i.i.d. Gaussian with mean 0 and variance $\sigma^2$. There
exists a constant $C=C(\alpha,\sigma,L,k,\beta,\Theta)>0$,
independent of $N$ and $d$, such that for any sequential algorithm 
using $N$ noisy queries and any estimator
$\hat{x}_N\in\Theta$ measurable w.r.t. $(z_i,y_i)_{i=1}^{N}$,  we have,
\[
\sup_{g\in\mathcal{F}'_{\alpha,\beta,k}(L,L_f,G,G_g,\bar{L})}
\mathbb{E}\bigl[f(\hat{x}_N)-f(x^\star)\bigr]
\geq C\,N^{-(\beta-1)/(\beta+k)}.
\]

\end{theorem}


\begin{remark}[Matching rates and dimension dependence]
\label{rem:matching}
Theorems~\ref{thm:upper-bound} and~\ref{thm:lower-bound-bumps} 
establish matching rates in $N$. Both scale as 
$N^{-(\beta-1)/(\beta+k)}$ for $\beta\geq 2$, thereby identifying 
the optimal rate in $N$ for the minimax optimization risk under 
noisy zero-order access to $g$. The dependence on the dimension 
$d$, however, is not sharp. 
The mismatch in $d$ is carried entirely by the prefactor 
$d^{(2\beta+k-1)/(\beta+k)}$ of the upper bound, whose sharpness 
in $d$ is left as an open question.
\end{remark}

\begin{remark}[The direct case $k=0$]
\label{rem:k0-consistency}
Our results also cover the direct zero-order optimization problem obtained
for $k=0$. In that case, the dimension exponent $(2\beta+k-1)/(\beta+k)$
reduces to $2-\tfrac{1}{\beta}$ and the rate in $N$ becomes $N^{-(\beta-1)/\beta}$,
which is the minimax optimal rate for direct zero-order
optimization of strongly convex functions~(\cite{polyak1990optimal}). 
The dimension exponent
$2-\tfrac{1}{\beta}$ is the same as in the upper bounds of~\citet{akhavan2021distributed,novitskii2021improved}, which is the best known result for $\beta>2$, except for the cases $\beta=2,3$, where the dimension exponent can be reduced to 1  (\cite{akhavan2020exploiting}, \cite{yu2024stochastic}). 
\end{remark}



\section{Proofs}
\label{sec:proofs}

\subsection{Proofs for the upper bound}
\label{proofs-upper-bounds}

\begin{proof}[Proof of Lemma~\ref{lem:bias-multi}]
Fix $l\in\{1,\dots,d\}$ and $x_t\in\mathbb{R}^d$. Since $\int K_1(v)\,dv=0$ and the noise variables $\varepsilon^\pm_{t,l}$ are
independent of $(U_{t,l},V_{t,l})$, the noise mean does not contribute to the
conditional expectation. Indeed,
\[
\mathbb{E}\bigl[K_k(U_{t,l})K_1(V_{t,l})(\varepsilon^+_{t,l}-\varepsilon^-_{t,l})\bigr]
=\mathbb{E}[K_k(U_{t,l})]\,\mathbb{E}[K_1(V_{t,l})]\,
\mathbb{E}[\varepsilon^+_{t,l}-\varepsilon^-_{t,l}]=0,
\]
since $\mathbb{E}[K_1(V_{t,l})]=0$. Therefore,
\begin{equation}
\label{eq:bias-Delta-form}
\mathbb{E}[\tilde{g}_{t,l}\mid x_t]
=\frac{1}{2h^{k+1}}\,\mathbb{E}\bigl[K_k(U_{t,l})\,K_1(V_{t,l})\,
\Delta_l(x_t,U_{t,l},V_{t,l})\bigr],
\end{equation}
where 
$$
\Delta_l(x,u,v):=g(x+hue_j+hve_l)-g(x+hue_j-hve_l).
$$

We now expand $g$ around $x_t$. Recall from Section~\ref{sec:prel} 
that $\ell_*=\lfloor\beta+k\rfloor=k+\lfloor\beta\rfloor$. Since 
$g\in\mathcal{F}_{\beta+k}(L)$, the H\"older 
property~\eqref{eq:holder} yields, for $(u,v)\in[-1/2,1/2]^2$,
\begin{equation}
\label{eq:bias-taylor}
g(x_t+hue_j\pm hve_l)
=\sum_{a+b\leq \ell_*}
\frac{h^{a+b}\,u^a\,(\pm v)^b}{a!\,b!}\,
\partial_j^a\partial_l^b g(x_t)+R_\pm(u,v),
\end{equation}
where the remainder satisfies
$$
|R_\pm(u,v)|\leq L\,\|hue_j\pm hve_l\|^{\beta+k}
\leq L\,h^{\beta+k}\,(|u|+|v|)^{\beta+k}.
$$
Indeed, if $l=j$ we use $|u\pm v|\leq |u|+|v|$, while for $l\neq j$ 
we use $\sqrt{u^2+v^2}\leq |u|+|v|$. Subtracting the two expansions 
in~\eqref{eq:bias-taylor} and noting that 
$(+v)^b-(-v)^b=2v^b$ for $b$ odd and $0$ for $b$ even, we get
\begin{equation}
\label{eq:bias-Delta-expansion}
\Delta_l(x_t,u,v)
=2\!\!\!\sum_{\substack{a+b\leq \ell_*\\b\;\mathrm{odd}}}\!\!
\frac{h^{a+b}\,u^a\,v^b}{a!\,b!}\,
\partial_j^a\partial_l^b g(x_t)
+R_+(u,v)-R_-(u,v).
\end{equation}

Substituting~\eqref{eq:bias-Delta-expansion} 
into~\eqref{eq:bias-Delta-form}, multiplying by 
$K_k(u)K_1(v)/(2h^{k+1})$, and using the independence of $U_{t,l}$ 
and $V_{t,l}$, we obtain
\begin{equation}
\label{eq:bias-after-int}
\mathbb{E}[\tilde{g}_{t,l}\mid x_t]
=\!\!\!\sum_{\substack{a+b\leq \ell_*\\b\;\mathrm{odd}}}\!\!
\frac{h^{a+b-k-1}}{a!\,b!}\,
\partial_j^a\partial_l^b g(x_t)\,\mu_k(a)\,\nu_1(b)
+\mathrm{Rem}_{t,l},
\end{equation}
where $\mu_k(a):=\int u^a K_k(u)\,du$, $\nu_1(b):=\int v^b 
K_1(v)\,dv$, and 
$$
\mathrm{Rem}_{t,l}
:=\mathbb{E}\!\left[\frac{K_k(U_{t,l})\,K_1(V_{t,l})}{2h^{k+1}}
\bigl(R_+(U_{t,l},V_{t,l})-R_-(U_{t,l},V_{t,l})\bigr)\right].
$$
We now apply the moment 
conditions~\eqref{eq:Kk-moments} and~\eqref{eq:K1-moments}. The 
condition~\eqref{eq:K1-moments} forces $\nu_1(b)=0$ for all odd 
$b\geq 3$, so only $b=1$ contributes 
in~\eqref{eq:bias-after-int}. Among the surviving terms, 
$\mu_k(a)=0$ for $a\neq k$ by~\eqref{eq:Kk-moments}. Thus the 
unique remaining pair is $(a,b)=(k,1)$, which lies in the 
admissible range $a+b=k+1\leq \ell_*$ since $\beta>1$ implies 
$\lfloor\beta\rfloor\geq 1$. Its contribution is
$$
\frac{h^{k+1-k-1}}{k!\cdot 1!}\,
\partial_j^k\partial_l g(x_t)\,
\mu_k(k)\,\nu_1(1)
=\partial_j^k\partial_l g(x_t),
$$
where we used $\mu_k(k)=k!$ and $\nu_1(1)=1$. By Schwarz's 
theorem~\eqref{eq:schwarz}, 
$\partial_j^k\partial_l g(x_t)=\partial_l f(x_t)$. We have shown
\begin{equation}
\label{eq:bias-decomp}
\mathbb{E}[\tilde{g}_{t,l}\mid x_t]=\partial_l f(x_t)
+\mathrm{Rem}_{t,l}.
\end{equation}

It remains to bound $\mathrm{Rem}_{t,l}$. Using the bound on 
$R_\pm$ and the identity $h^{\beta+k}/h^{k+1}=h^{\beta-1}$,
$$
|\mathrm{Rem}_{t,l}|
\leq L\,h^{\beta-1}\,
\mathbb{E}\bigl[|K_k(U_{t,l})|\,|K_1(V_{t,l})|\,
(|U_{t,l}|+|V_{t,l}|)^{\beta+k}\bigr].
$$
Since $U_{t,l},V_{t,l}\in[-1/2,1/2]$, we have 
$|U_{t,l}|+|V_{t,l}|\leq 1$, and the last factor in the expectation 
is bounded by $1$. Applying the Cauchy--Schwarz inequality together 
with the independence of $U_{t,l}$ and $V_{t,l}$,
\begin{equation}
\label{eq:bias-Rem-bound}
|\mathrm{Rem}_{t,l}|
\leq L\,h^{\beta-1}
\bigl(\mathbb{E}[K_k(U_{t,l})^2]\bigr)^{1/2}
\bigl(\mathbb{E}[K_1(V_{t,l})^2]\bigr)^{1/2}
=C_{\mathrm{bias}}\,h^{\beta-1}.
\end{equation}
Combining~\eqref{eq:bias-decomp} and~\eqref{eq:bias-Rem-bound} 
yields the coordinatewise bound. Squaring and summing over 
$l=1,\dots,d$ completes the proof.
\end{proof}

\begin{proof}[Proof of Lemma~\ref{lem:var-multi-second-moment}]
\label{proof:lemma2}
We have
$$
\mathbb{E}\bigl[\|\tilde{g}_t\|^2\mid x_t\bigr]
\leq 2\bigl\|\mathbb{E}[\tilde{g}_t\mid x_t]\bigr\|^2
+2\,\mathbb{E}\bigl[\|\tilde{g}_t-\mathbb{E}[\tilde{g}_t\mid x_t]\|^2
\mid x_t\bigr].
$$
We bound 
the two summands on the right hand side of this inequality separately.

\medskip
\noindent\textit{Bound on the squared conditional mean.}
By the same elementary inequality applied with 
$a=\mathbb{E}[\tilde{g}_t\mid x_t]$ and $b=\nabla f(x_t)$, together 
with Lemma~\ref{lem:bias-multi} and the assumption 
$\|\nabla f(x_t)\|\leq G$ on $\Theta$,
$$
\bigl\|\mathbb{E}[\tilde{g}_t\mid x_t]\bigr\|^2
\leq 2\|\nabla f(x_t)\|^2
+2\bigl\|\mathbb{E}[\tilde{g}_t\mid x_t]-\nabla f(x_t)\bigr\|^2
\leq 2G^2+2d\,C_{\mathrm{bias}}^2\,h^{2(\beta-1)}.
$$
Hence the first summand is bounded by 
$4G^2+4d\,C_{\mathrm{bias}}^2\,h^{2(\beta-1)}$.

\medskip
\noindent\textit{Bound on the conditional variance.}
The random variables 
$\{(U_{t,l},V_{t,l},\varepsilon^\pm_{t,l})\}_{l=1}^d$ are mutually 
independent, and so are the coordinates 
$\tilde{g}_{t,l}$ conditionally on $x_t$. Consequently,
$$
\mathbb{E}\bigl[\|\tilde{g}_t-\mathbb{E}[\tilde{g}_t\mid x_t]\|^2
\mid x_t\bigr]
=\sum_{l=1}^d\mathrm{Var}(\tilde{g}_{t,l}\mid x_t)
\leq\sum_{l=1}^d\mathbb{E}[\tilde{g}_{t,l}^2\mid x_t],
$$
and it suffices to bound $\mathbb{E}[\tilde{g}_{t,l}^2\mid x_t]$ for 
each $l$.

We condition on $(U_{t,l},V_{t,l})$. Setting $x:=x_t+hU_{t,l}e_j$ 
and $s:=hV_{t,l}e_l$ (so that $\|s\|\leq h/2$),
\begin{equation}
\label{eq:var-tower}
\mathbb{E}[\tilde{g}_{t,l}^2\mid x_t]
=\frac{1}{4h^{2(k+1)}}\,\mathbb{E}\Bigl[
K_k(U_{t,l})^2\,K_1(V_{t,l})^2\,
\mathbb{E}\bigl[(y^+_{t,l}-y^-_{t,l})^2
\,\big|\,U_{t,l},V_{t,l},x_t\bigr]\Bigr].
\end{equation}
Since $\varepsilon^\pm_{t,l}$ are independent of $(U_{t,l},V_{t,l})$ and satisfy
$\mathbb{E}[(\varepsilon^\pm_{t,l})^2]\leq\sigma^2$, the inequality
$(a+b)^2\leq 2a^2+2b^2$ applied to
$y^+_{t,l}-y^-_{t,l}=(g(x+s)-g(x-s))+(\varepsilon^+_{t,l}-\varepsilon^-_{t,l})$
gives
\begin{equation}
\label{eq:var-noise-split}
\mathbb{E}\bigl[(y^+_{t,l}-y^-_{t,l})^2\,\big|\,U_{t,l},V_{t,l},x_t\bigr]
\leq 2\bigl(g(x+s)-g(x-s)\bigr)^2
+2\,\mathbb{E}[(\varepsilon^+_{t,l}-\varepsilon^-_{t,l})^2]
\leq 2\bigl(g(x+s)-g(x-s)\bigr)^2+8\sigma^2,
\end{equation}
where we used $\mathbb{E}[(\varepsilon^+_{t,l}-\varepsilon^-_{t,l})^2]
\leq 2\mathbb{E}[(\varepsilon^+_{t,l})^2]+2\mathbb{E}[(\varepsilon^-_{t,l})^2]\leq 4\sigma^2$.

We now bound $|g(x+s)-g(x-s)|$ using the smoothness assumptions on 
$g$. Decompose 
\begin{equation*}
g(x+s)-g(x-s)=P+Q+R,
\end{equation*}
where 
\begin{align*}
P&:= g(x+s)-g(x)-\langle\nabla g(x),s\rangle, \\
Q&:= -\bigl(g(x-s)-g(x)+\langle\nabla g(x),s\rangle\bigr), \\
R&:= 2\langle\nabla g(x),s\rangle.
\end{align*}
Since $\nabla g$ is $\bar{L}$-Lipschitz on $\Theta^+$ and 
$x,x\pm s\in\Theta^+$ (as $\|s\|\leq h/2\leq 1/2$, so $x\pm s$ lies 
within distance $1$ of $\Theta$),
$$
|P|=\left|\int_0^1\langle\nabla g(x+\tau s)-\nabla g(x),s\rangle
\,d\tau\right|
\leq\int_0^1 \bar{L}\tau\|s\|^2\,d\tau
=\frac{\bar{L}\|s\|^2}{2},
$$
and similarly $|Q|\leq\bar{L}\|s\|^2/2$. By Cauchy--Schwarz and the 
bound $\|\nabla g(x)\|\leq G_g$ on $\Theta^+$,
$$
|R|=2|\langle\nabla g(x),s\rangle|\leq 2G_g\|s\|.
$$
The inequality $(a+b+c)^2\leq 3(a^2+b^2+c^2)$ then gives
$$
\bigl(g(x+s)-g(x-s)\bigr)^2
\leq 3\Bigl(2\cdot\tfrac{\bar{L}^2\|s\|^4}{4}+4G_g^2\|s\|^2\Bigr)
=\frac{3\bar{L}^2\|s\|^4}{2}+12G_g^2\|s\|^2.
$$
Using $\|s\|\leq h/2$,
\begin{equation}
\label{eq:var-g-diff}
\bigl(g(x+s)-g(x-s)\bigr)^2
\leq\frac{3\bar{L}^2 h^4}{32}+3G_g^2 h^2.
\end{equation}

Substituting~\eqref{eq:var-noise-split} 
and~\eqref{eq:var-g-diff} into~\eqref{eq:var-tower}, and using the 
independence of $U_{t,l}$ and $V_{t,l}$,
\begin{equation*}
\mathbb{E}[\tilde{g}_{t,l}^2\mid x_t]
\leq\frac{\mathbb{E}[K_k(U)^2]\,\mathbb{E}[K_1(V)^2]}{4h^{2(k+1)}}
\Bigl(\frac{3\bar{L}^2 h^4}{16}+6G_g^2 h^2+8\sigma^2\Bigr).
\end{equation*}
Since $h \le 1$, we have $h^4 \le 1$ and $h^2 \le 1$. Hence, the
parenthesis is bounded by a constant independent of $h$, namely
\[
\frac{3\bar{L}^2}{16}+6G_g^2+8\sigma^2 .
\]
Therefore,
\begin{equation}
\label{eq:var-coord-bound}
\mathbb{E}[\tilde{g}_{t,l}^2\mid x_t]
\leq C_{\mathrm{var}}\,h^{-2(k+1)}.
\end{equation}

Summing~\eqref{eq:var-coord-bound} over $l=1,\dots,d$ and combining 
with the bound on the squared conditional mean completes the proof.
\end{proof}

\medskip
The next lemma provides a quantitative recursive bound that we will 
apply in the proof of Theorem~\ref{thm:upper-bound}. It is a 
non-asymptotic version of a classical result going back 
to~\citet{chung1954stochastic}. Closely related non-asymptotic 
formulations can be found in~\citet{moulines2011non} 
and~\citet{jiang2024generalized}. The statement below treats 
simultaneously two forcing terms of different polynomial orders, 
which arise naturally in the analysis of gradient-free methods 
driven by biased gradient estimators.

\begin{lemma}[Recursive inequality]
\label{lem:recursion}
Let $(u_t)_{t\geq 0}$ be a sequence of non-negative real numbers 
satisfying
\begin{equation}
\label{eq:recursion-hypo}
u_{t+1}\leq\Bigl(1-\frac{q}{t+t_0}\Bigr)u_t
+\frac{A}{t+t_0}+\frac{B}{(t+t_0)^2},\qquad t\geq 0,
\end{equation}
for some constants $q\geq 2$, $t_0\geq 2q$ and $A,B\geq 0$. Then, 
for all $T\geq 1$,
\begin{equation}
\label{eq:recursion-conclusion}
u_T\leq\frac{t_0\,u_0}{T+t_0}+\frac{2A}{q}
+\frac{2B}{(q-1)(T+t_0)}.
\end{equation}
\end{lemma}

\begin{proof}
Set $w_{s,t}:=\prod_{r=s}^{t-1}\bigl(1-q/(r+t_0)\bigr)$ for 
$0\leq s\leq t$, with the convention $w_{t,t}:=1$. Since 
$t_0\geq 2q$, all factors satisfy $1-q/(r+t_0)\geq 1/2$, hence are 
non-negative. Iterating~\eqref{eq:recursion-hypo} yields
\begin{equation}
\label{eq:rec-unroll}
u_T\leq w_{0,T}\,u_0
+\sum_{t=0}^{T-1}\frac{w_{t+1,T}\,A}{t+t_0}
+\sum_{t=0}^{T-1}\frac{w_{t+1,T}\,B}{(t+t_0)^2}.
\end{equation}
We bound each term separately.

\medskip
\noindent\textit{Bound on the weights and on the initialization term.}
Using $1-x\leq e^{-x}$ for $x\in[0,1]$ together with the integral 
inequality $\sum_{r=s}^{t-1}1/(r+t_0)\geq\log\bigl((t+t_0)/(s+t_0)\bigr)$,
\begin{equation}
\label{eq:rec-weight}
w_{s,t}\leq\Bigl(\frac{s+t_0}{t+t_0}\Bigr)^{q},\qquad 0\leq s\leq t.
\end{equation}
Applied with $s=0$ and combined with $q\geq 1$ and 
$t_0/(T+t_0)\leq 1$, this yields
\begin{equation}
\label{eq:rec-init}
w_{0,T}\,u_0
\leq\Bigl(\frac{t_0}{T+t_0}\Bigr)^q u_0
\leq\frac{t_0\,u_0}{T+t_0}.
\end{equation}
For the two summation terms, we will repeatedly use the shift 
inequality $(t+1+t_0)^q\leq 2\,(t+t_0)^q$, valid for all $t\geq 0$. 
This follows from $(t+1+t_0)/(t+t_0)\leq 1+1/t_0\leq 1+1/(2q)$ 
(by $t_0\geq 2q$) together with $(1+1/(2q))^q\leq e^{1/2}\leq 2$ 
(applying $(1+x/q)^q\leq e^x$ with $x=1/2$).

\medskip
\noindent\textit{Bound on the $A$-term.}
Applying~\eqref{eq:rec-weight} with $s=t+1$ and the shift 
inequality,
\[
\sum_{t=0}^{T-1}\frac{w_{t+1,T}\,A}{t+t_0}
\leq\frac{A}{(T+t_0)^q}\sum_{t=0}^{T-1}\frac{(t+1+t_0)^q}{t+t_0}
\leq\frac{2A}{(T+t_0)^q}\sum_{t=0}^{T-1}(t+t_0)^{q-1}.
\]
Since $q\geq 2$, the function $r\mapsto(r+t_0)^{q-1}$ is 
non-decreasing, so by the integral test
\[
\sum_{t=0}^{T-1}(t+t_0)^{q-1}
\leq\int_0^T(r+t_0)^{q-1}\,dr
=\frac{(T+t_0)^q-t_0^q}{q}
\leq\frac{(T+t_0)^q}{q}.
\]
Therefore
\begin{equation}
\label{eq:rec-A-bound}
\sum_{t=0}^{T-1}\frac{w_{t+1,T}\,A}{t+t_0}\leq\frac{2A}{q}.
\end{equation}

\medskip
\noindent\textit{Bound on the $B$-term.}
Similarly,
\[
\sum_{t=0}^{T-1}\frac{w_{t+1,T}\,B}{(t+t_0)^2}
\leq\frac{B}{(T+t_0)^q}\sum_{t=0}^{T-1}\frac{(t+1+t_0)^q}{(t+t_0)^2}
\leq\frac{2B}{(T+t_0)^q}\sum_{t=0}^{T-1}(t+t_0)^{q-2}.
\]
Since $q\geq 2$, the function $r\mapsto(r+t_0)^{q-2}$ is 
non-decreasing, and by the integral test
\[
\sum_{t=0}^{T-1}(t+t_0)^{q-2}
\leq\int_0^T(r+t_0)^{q-2}\,dr
=\frac{(T+t_0)^{q-1}-t_0^{q-1}}{q-1}
\leq\frac{(T+t_0)^{q-1}}{q-1}.
\]
Therefore
\begin{equation}
\label{eq:rec-B-bound}
\sum_{t=0}^{T-1}\frac{w_{t+1,T}\,B}{(t+t_0)^2}
\leq\frac{2B}{(q-1)(T+t_0)}.
\end{equation}

Combining~\eqref{eq:rec-init}, \eqref{eq:rec-A-bound} 
and~\eqref{eq:rec-B-bound} into~\eqref{eq:rec-unroll} 
yields~\eqref{eq:recursion-conclusion}.
\end{proof}

\begin{remark}
\label{rem:recursion}
The two forcing terms in~\eqref{eq:recursion-hypo} contribute 
qualitatively differently to the bound~\eqref{eq:recursion-conclusion}. 
The $A$-term gives a residual contribution $2A/q$ that does not 
decay with $T$, whereas the $B$-term decays as $1/(T+t_0)$.  
In the proof of
Theorem~\ref{thm:upper-bound} below, the $A$-term encodes the squared 
bias of the gradient estimator. 
\end{remark}

\begin{proof}[Proof of Theorem~\ref{thm:upper-bound}]
\label{proof:thm1}
Throughout the proof, fix 
$g\in\mathcal{F}'_{\alpha,\beta,k}(L,L_f,G,G_g,\bar{L})$ and 
let $x^\star\in\mathrm{int}(\Theta)$ denote the minimizer of 
$f=\partial_j^k g$ on $\Theta$, so that $\nabla f(x^\star)=0$. The 
recursion~\eqref{eq:sgd} guarantees $x_t\in\Theta$ for all 
$t\geq 0$, hence properties (b), (c), (d) of the class apply at 
every iteration. Set
\[
\Delta_t:=\mathbb{E}\bigl[\|x_t-x^\star\|^2\bigr],
\qquad 
b_t:=\mathbb{E}[\tilde{g}_t\mid x_t]-\nabla f(x_t),
\]
and note that $\Delta_0\leq R^2$.

\medskip
\noindent\textit{One-step recursive inequality on $\Delta_t$.}
Since $\Pi_\Theta$ is non-expansive and $x^\star\in\Theta$,
\[
\|x_{t+1}-x^\star\|^2
\leq\|x_t-\eta_t\tilde{g}_t-x^\star\|^2
=\|x_t-x^\star\|^2-2\eta_t\langle\tilde{g}_t,x_t-x^\star\rangle
+\eta_t^2\|\tilde{g}_t\|^2.
\]
Taking the conditional expectation given $x_t$, using 
$\mathbb{E}[\tilde{g}_t\mid x_t]=\nabla f(x_t)+b_t$, then taking 
total expectation,
\begin{equation}
\label{eq:upper-rec1}
\Delta_{t+1}\leq\Delta_t
-2\eta_t\,\mathbb{E}\bigl[\langle\nabla f(x_t),x_t-x^\star\rangle\bigr]
-2\eta_t\,\mathbb{E}\bigl[\langle b_t,x_t-x^\star\rangle\bigr]
+\eta_t^2\,\mathbb{E}\bigl[\|\tilde{g}_t\|^2\bigr].
\end{equation}
By $\alpha$-strong convexity of $f$ on $\Theta$ and 
$\nabla f(x^\star)=0$,
\begin{equation}
\label{eq:upper-sc}
\langle\nabla f(x_t),x_t-x^\star\rangle
\geq f(x_t)-f(x^\star)+\frac{\alpha}{2}\|x_t-x^\star\|^2
\geq\frac{\alpha}{2}\|x_t-x^\star\|^2.
\end{equation}
By the Cauchy--Schwarz inequality and Young's inequality 
$2|\langle u,v\rangle|\leq\varepsilon\|u\|^2+\|v\|^2/\varepsilon$ 
applied with $\varepsilon=\alpha/2$,
\begin{equation}
\label{eq:upper-bias}
2\bigl|\langle b_t,x_t-x^\star\rangle\bigr|
\leq\frac{\alpha}{2}\|x_t-x^\star\|^2+\frac{2}{\alpha}\|b_t\|^2.
\end{equation}
Substituting~\eqref{eq:upper-sc} and~\eqref{eq:upper-bias} 
into~\eqref{eq:upper-rec1},
\begin{equation}
\label{eq:upper-rec2}
\Delta_{t+1}\leq\Bigl(1-\frac{\alpha\eta_t}{2}\Bigr)\Delta_t
+\frac{2\eta_t}{\alpha}\,\mathbb{E}\bigl[\|b_t\|^2\bigr]
+\eta_t^2\,\mathbb{E}\bigl[\|\tilde{g}_t\|^2\bigr].
\end{equation}

\medskip
\noindent\textit{Reduction to a Chung-type recursion.}
By Lemma~\ref{lem:bias-multi}, 
\begin{equation}
\label{eq:upper-bias-bound}
\mathbb{E}\bigl[\|b_t\|^2\bigr]
\leq d\,C_{\mathrm{bias}}^2\,h^{2(\beta-1)},
\end{equation}
and by Lemma~\ref{lem:var-multi-second-moment}, applicable since 
$x_t\in\Theta$ and $h\in(0,1]$,
\begin{equation}
\label{eq:upper-var-bound}
\mathbb{E}\bigl[\|\tilde{g}_t\|^2\bigr]
\leq 4G^2+4d\,C_{\mathrm{bias}}^2\,h^{2(\beta-1)}
+2d\,C_{\mathrm{var}}\,h^{-2(k+1)}.
\end{equation}
Substituting~\eqref{eq:upper-bias-bound} 
and~\eqref{eq:upper-var-bound} into~\eqref{eq:upper-rec2} with 
$\eta_t=\gamma/(t+t_0)$ yields
\begin{equation}
\label{eq:upper-rec3}
\Delta_{t+1}\leq\Bigl(1-\frac{q}{t+t_0}\Bigr)\Delta_t
+\frac{A}{t+t_0}+\frac{B}{(t+t_0)^2},
\end{equation}
where 
\begin{equation}
\label{eq:upper-qAB}
q:=\frac{\alpha\gamma}{2},\quad
A:=\frac{2\gamma\,d\,C_{\mathrm{bias}}^2\,h^{2(\beta-1)}}{\alpha},\quad
B:=\gamma^2\bigl(4G^2+4d\,C_{\mathrm{bias}}^2\,h^{2(\beta-1)}
+2d\,C_{\mathrm{var}}\,h^{-2(k+1)}\bigr).
\end{equation}
The conditions $\gamma\geq 4/\alpha$ and $t_0\geq\alpha\gamma$ 
imposed in the statement of the theorem are equivalent to $q\geq 2$ 
and $t_0\geq 2q$, which are precisely the hypotheses of 
Lemma~\ref{lem:recursion}. Applying this lemma 
to~\eqref{eq:upper-rec3} with $u_t=\Delta_t$ and $u_0\leq R^2$ 
yields, for all $T\geq 1$,
\begin{equation}
\label{eq:upper-DeltaT-from-lemma}
\Delta_T\leq\frac{t_0\,R^2}{T+t_0}
+\frac{2A}{q}
+\frac{2B}{(q-1)(T+t_0)}.
\end{equation}

\medskip
\noindent\textit{Bias-variance bound on $\Delta_T$.}
Using $q=\alpha\gamma/2$ in $2A/q$, and $T+t_0\geq T$ in the third 
term of~\eqref{eq:upper-DeltaT-from-lemma},
\[
\frac{2A}{q}=\frac{8\,d\,C_{\mathrm{bias}}^2\,h^{2(\beta-1)}}{\alpha^2},
\qquad
\frac{2B}{(q-1)(T+t_0)}\leq\frac{2B}{(q-1)T}.
\]
Substituting the expression of $B$ from~\eqref{eq:upper-qAB} and 
grouping terms,
\begin{equation}
\label{eq:upper-DeltaT-grouped}
\Delta_T\leq\frac{C_3}{T}+\frac{C_4\,d\,h^{-2(k+1)}}{T}
+C_5\,d\,h^{2(\beta-1)},
\end{equation}
where 
\begin{equation}
\label{eq:upper-C345}
C_3:=t_0\,R^2+\frac{8\gamma^2 G^2}{q-1},\quad
C_4:=\frac{4\gamma^2\,C_{\mathrm{var}}}{q-1},\quad
C_5:=\frac{8\,C_{\mathrm{bias}}^2}{\alpha^2}
+\frac{8\gamma^2\,C_{\mathrm{bias}}^2}{q-1}.
\end{equation}
The contribution 
$8\gamma^2\,d\,C_{\mathrm{bias}}^2\,h^{2(\beta-1)}/((q-1)T)$ 
arising from the bias-related part of 
$2B/((q-1)T)$ has been absorbed into the term 
$C_5\,d\,h^{2(\beta-1)}$, using $1/T\leq 1$ for $T\geq 1$.

\medskip
\noindent\textit{Conversion to optimization error.}
By the smoothness inequality~\eqref{eq:smoothness} (valid since 
$x_T\in\Theta$ and $\nabla f(x^\star)=0$),
\[
\mathbb{E}\bigl[f(x_T)-f(x^\star)\bigr]
\leq\frac{L_f}{2}\,\Delta_T,
\]
which combined with~\eqref{eq:upper-DeltaT-grouped} gives
\begin{equation}
\label{eq:upper-opt-error}
\mathbb{E}\bigl[f(x_T)-f(x^\star)\bigr]
\leq\frac{C_6}{T}+\frac{C_7\,d\,h^{-2(k+1)}}{T}
+C_8\,d\,h^{2(\beta-1)},
\end{equation}
where $C_i:=L_f\,C_{i-3}/2$ for $i\in\{6,7,8\}$.

\medskip
\noindent\textit{Optimal choice of $h$.}
We now substitute the prescribed value 
$h=\min(\kappa\,T^{-1/(2(\beta+k))},1)$ and consider two regimes.

If $T\geq\kappa^{2(\beta+k)}$, then 
$h=\kappa\,T^{-1/(2(\beta+k))}\leq 1$, so
\[
h^{-2(k+1)}=\kappa^{-2(k+1)}\,T^{(k+1)/(\beta+k)},
\qquad
h^{2(\beta-1)}=\kappa^{2(\beta-1)}\,T^{-(\beta-1)/(\beta+k)}.
\]
Since $1-(k+1)/(\beta+k)=(\beta-1)/(\beta+k)$, the second and 
third terms in~\eqref{eq:upper-opt-error} both scale as 
$T^{-(\beta-1)/(\beta+k)}$. Moreover, $(\beta-1)/(\beta+k)<1$ 
(since $\beta>1$ and $k\geq 0$) yields 
$T^{-1}\leq T^{-(\beta-1)/(\beta+k)}$, so the initialization term 
$C_6/T$ is also dominated by $T^{-(\beta-1)/(\beta+k)}$. Combining 
and using $d\geq 1$,
\begin{equation}
\label{eq:upper-regime1}
\mathbb{E}\bigl[f(x_T)-f(x^\star)\bigr]
\leq C_9\,d\,T^{-(\beta-1)/(\beta+k)},
\end{equation}
where 
$C_9:=C_6+C_7\,\kappa^{-2(k+1)}+C_8\,\kappa^{2(\beta-1)}$.

If $1\leq T<\kappa^{2(\beta+k)}$, then $h=1$, so 
$h^{2(\beta-1)}=h^{-2(k+1)}=1$ 
and~\eqref{eq:upper-opt-error} reduces to 
$\mathbb{E}[f(x_T)-f(x^\star)]\leq(C_6+C_7\,d)/T+C_8\,d$. Since 
$T\leq\kappa^{2(\beta+k)}$, 
$T^{-(\beta-1)/(\beta+k)}\geq\kappa^{-2(\beta-1)}$. Hence, setting
\[
C_9':=\max\bigl(\kappa^{2(\beta-1)}(C_6+C_7+C_8),\;C_9\bigr),
\]
the bound~\eqref{eq:upper-regime1}, with $C_9$ replaced by $C_9'$, 
also holds in this regime. In both cases,
\begin{equation}
\label{eq:upper-final-T}
\mathbb{E}\bigl[f(x_T)-f(x^\star)\bigr]
\leq C_0\,d\,T^{-(\beta-1)/(\beta+k)},
\end{equation}
where $C_0:=C_9'$ depends only on 
$\alpha,\gamma,t_0,\kappa,L_f,G,R,C_{\mathrm{bias}},C_{\mathrm{var}}$.

\medskip
\noindent\textit{Conversion to total oracle budget $N$.}
Recall $N=2dT$, so $T=N/(2d)$. Hence
\[
T^{-(\beta-1)/(\beta+k)}
=2^{(\beta-1)/(\beta+k)}\,d^{(\beta-1)/(\beta+k)}\,
N^{-(\beta-1)/(\beta+k)},
\]
and using $1+(\beta-1)/(\beta+k)=(2\beta+k-1)/(\beta+k)$,
\[
d\,T^{-(\beta-1)/(\beta+k)}
=2^{(\beta-1)/(\beta+k)}\,d^{(2\beta+k-1)/(\beta+k)}\,
N^{-(\beta-1)/(\beta+k)}.
\]
Setting $C:=2^{(\beta-1)/(\beta+k)}\,C_0$ and combining 
with~\eqref{eq:upper-final-T},
\[
\mathbb{E}\bigl[f(x_T)-f(x^\star)\bigr]
\leq C\,d^{(2\beta+k-1)/(\beta+k)}\,N^{-(\beta-1)/(\beta+k)}.
\]
This bound is uniform over 
$g\in\mathcal{F}'_{\alpha,\beta,k}(L,L_f,G,G_g,\bar{L})$, which 
completes the proof.
\end{proof}

\subsection{Proofs for the lower bound}

\begin{proof}[Proof of Theorem~\ref{thm:lower-bound-bumps}]

By translating the coordinates, we can assume that 
$0\in\mathrm{int}(\Theta)$. Let $h_0\in(0,1]$ be such that
\begin{equation}
\label{eq:lb-Theta-cube}
[-2h_0,2h_0]^d\subset\mathrm{int}(\Theta).
\end{equation}
For notational simplicity we write the proof for the coordinate 
direction $j=1$. The general case is obtained by replacing 
$x_1,e_1,\partial_1$ throughout by $x_j,e_j,\partial_j$ and the 
sum over $l=2,\dots,d$ by the sum over $l\neq j$. We denote 
$R_\Theta:=\sup_{x\in\Theta}\|x\|$.

\medskip
\noindent\textit{Construction of the two hypotheses.}
Let $\Phi\in C_c^\infty(\mathbb{R})$ be supported on $[-1,1]$ and 
satisfy
\begin{equation}
\label{eq:lb-Phi}
\Phi^{(k+1)}(0)\neq 0.
\end{equation}
Such a $\Phi$ exists. For instance, take 
$\psi\in C_c^\infty(\mathbb{R})$ supported on $[-1,1]$ with 
$\psi(0)=1/(k+1)!$ and set $\Phi(t):=t^{k+1}\psi(t)$. Then 
Leibniz's formula gives $\Phi^{(k+1)}(0)=1$.

Define the polynomial
\begin{equation}
\label{eq:lb-q}
q(x):=\frac{2\alpha}{(k+2)!}\,x_1^{k+2}
+\frac{\alpha}{k!}\sum_{l=2}^d x_l^2\,x_1^k.
\end{equation}
A direct computation shows
\begin{equation}
\label{eq:lb-q-derivative}
\partial_1^k q(x)=\alpha\,x_1^2+\alpha\sum_{l=2}^d x_l^2
=\alpha\,\|x\|^2.
\end{equation}
For parameters $h\in(0,h_0]$ and $r>0$ to be chosen later, define
\begin{equation}
\label{eq:lb-gpm}
g_\pm(x):=q(x)\pm r\,h^{\beta+k}\,\Phi(x_1/h),
\end{equation}
and set 
$f_\pm(x):=\partial_1^k g_\pm(x)
=\alpha\,\|x\|^2\pm r\,h^\beta\,\Phi^{(k)}(x_1/h)$.

\medskip
\noindent\textit{Verification of property (a): $g_\pm\in\mathcal{F}_{\beta+k}(L)$.}
Let $\ell_*$ denote the largest integer strictly less than 
$\beta+k$. Since $\beta\geq 2$, we distinguish two cases.

\emph{Case $\beta>2$.} Then $\beta+k>k+2$ so $\ell_*\geq k+2$. 
Since $q$ is a polynomial of degree $k+2$, its Taylor expansion of 
order $\ell_*$ is exact, and consequently 
$q\in\mathcal{F}_{\beta+k}(0)$.

\emph{Case $\beta=2$.} Then $\beta+k=k+2$ is an integer, so 
$\ell_*=\beta+k-1=k+1$. The Taylor remainder of order $\ell_*$ of 
$q$ is 
\[
R_q(x,z)=\sum_{|m|=k+2}\frac{D^m q(x)}{m!}(z-x)^m.
\]
The only nonzero derivatives of order $k+2$ of $q$ are 
$\partial_1^{k+2}q=2\alpha$ and 
$\partial_1^k\partial_l^2 q=2\alpha$ for $l\in\{2,\dots,d\}$. 
Hence
\[
R_q(x,z)
=\frac{2\alpha}{(k+2)!}(z_1-x_1)^{k+2}
+\frac{\alpha}{k!}(z_1-x_1)^k\sum_{l=2}^d(z_l-x_l)^2.
\]
Using $|z_1-x_1|\leq\|z-x\|$ and 
$\sum_{l=2}^d(z_l-x_l)^2\leq\|z-x\|^2$,
\[
|R_q(x,z)|
\leq\frac{2\alpha}{(k+2)!}\|z-x\|^{k+2}
+\frac{\alpha}{k!}\|z-x\|^{k+2}
=C_q\,\alpha\,\|z-x\|^{k+2},
\]
where 
\begin{equation}
\label{eq:lb-Cq-def}
C_q:=\frac{2}{(k+2)!}+\frac{1}{k!}.
\end{equation}
Since $k+2=\beta+k$ in this case, this is exactly the H\"older 
property at exponent $\beta+k$, so 
$q\in\mathcal{F}_{\beta+k}(C_q\,\alpha)$.

In both cases, $q\in\mathcal{F}_{\beta+k}(C_q\,\alpha)$, with the 
bound being an overestimate when $\beta>2$.

We now treat the bump term 
$b_h(x):=r\,h^{\beta+k}\,\Phi(x_1/h)$. Since $b_h$ depends only 
on $x_1$, its multivariate H\"older regularity reduces to the 
univariate one for 
$\varphi_h(t):=r\,h^{\beta+k}\,\Phi(t/h)$. Setting 
$s:=\beta+k-\ell_*\in(0,1]$, we have 
$\varphi_h^{(\ell_*)}(t)=r\,h^s\,\Phi^{(\ell_*)}(t/h)$, so
\begin{equation}
\label{eq:lb-phi-lstar}
\bigl|\varphi_h^{(\ell_*)}(a)-\varphi_h^{(\ell_*)}(b)\bigr|
=r\,h^s\,\bigl|\Phi^{(\ell_*)}(a/h)-\Phi^{(\ell_*)}(b/h)\bigr|.
\end{equation}
We control the right-hand side of~\eqref{eq:lb-phi-lstar} 
by considering two cases. If $|a-b|\leq h$, the mean value 
theorem gives 
$|\Phi^{(\ell_*)}(a/h)-\Phi^{(\ell_*)}(b/h)|\leq 
\|\Phi^{(\ell_*+1)}\|_\infty\,|a-b|/h$, hence
\[
\bigl|\varphi_h^{(\ell_*)}(a)-\varphi_h^{(\ell_*)}(b)\bigr|
\leq r\,\|\Phi^{(\ell_*+1)}\|_\infty\,h^{s-1}|a-b|
\leq r\,\|\Phi^{(\ell_*+1)}\|_\infty\,|a-b|^s,
\]
using $h^{s-1}|a-b|\leq |a-b|^s$ when $|a-b|\leq h$ and 
$s\leq 1$. If $|a-b|>h$, then 
$|\Phi^{(\ell_*)}(a/h)-\Phi^{(\ell_*)}(b/h)|
\leq 2\|\Phi^{(\ell_*)}\|_\infty$ and $h^s\leq |a-b|^s$, so
\[
\bigl|\varphi_h^{(\ell_*)}(a)-\varphi_h^{(\ell_*)}(b)\bigr|
\leq 2r\,\|\Phi^{(\ell_*)}\|_\infty\,|a-b|^s.
\]
Combining, there exists $C_\Phi>0$ (depending on $\Phi,\beta,k$ 
only) such that
\begin{equation}
\label{eq:lb-phi-holder-univar}
\bigl|\varphi_h^{(\ell_*)}(a)-\varphi_h^{(\ell_*)}(b)\bigr|
\leq C_\Phi\,r\,|a-b|^s,\qquad\forall\,a,b\in\mathbb{R}.
\end{equation}
The standard equivalence between H\"older continuity of the 
$\ell_*$-th derivative and the Taylor-remainder definition of the 
H\"older class 
implies that, after enlarging $C_\Phi$ by a constant depending 
only on $\beta+k$, we have 
$\varphi_h\in\mathcal{F}_{\beta+k}(C_\Phi\,r)$ uniformly in 
$h\in(0,h_0]$. Since $|x_1-z_1|\leq\|x-z\|$ in the multivariate 
case, this also yields
\begin{equation}
\label{eq:lb-bh-holder}
b_h\in\mathcal{F}_{\beta+k}(C_\Phi\,r).
\end{equation}

Combining the bounds for $q$ and $b_h$, we obtain 
$g_\pm=q\pm b_h\in\mathcal{F}_{\beta+k}(C_q\alpha+C_\Phi r)$. 
Setting 
\begin{equation}
\label{eq:lb-L0}
L_0:=C_q\,\alpha=\Bigl(\frac{2}{(k+2)!}+\frac{1}{k!}\Bigr)\alpha,
\end{equation}
the assumption $L>L_0$ allows us to define 
\begin{equation}
\label{eq:lb-rL}
r_L:=(L-L_0)/C_\Phi>0,
\end{equation}
and the condition $r\leq r_L$ guarantees 
$g_\pm\in\mathcal{F}_{\beta+k}(L)$.

\medskip
\noindent\textit{Verification of properties (b)--(d): strong convexity, Lipschitz gradient, bounded gradient of $f_\pm$ on $\Theta$.}
A direct computation gives
\begin{equation}
\label{eq:lb-Hessfpm}
\nabla^2 f_\pm(x)=2\alpha\,I_d
\pm r\,h^{\beta-2}\,\Phi^{(k+2)}(x_1/h)\,e_1 e_1^\top.
\end{equation}
The operator norm of the perturbation is at most 
$r\,h^{\beta-2}\|\Phi^{(k+2)}\|_\infty$. Since $\beta\geq 2$ and 
$h\leq h_0\leq 1$, we have $h^{\beta-2}\leq 1$. Setting
\begin{equation}
\label{eq:lb-rsc}
r_{\mathrm{sc}}:=\frac{\alpha}{1+\|\Phi^{(k+2)}\|_\infty},
\end{equation}
the condition $r\leq r_{\mathrm{sc}}$ yields a perturbation norm 
at most $\alpha$, hence $\nabla^2 f_\pm(x)\succeq\alpha\,I_d$ on 
$\mathbb{R}^d$. Therefore $f_\pm$ is $\alpha$-strongly convex on 
$\mathbb{R}^d$, in particular on $\Theta$, which proves~(b).

Under the same condition $r\leq r_{\mathrm{sc}}$,
\[
\|\nabla^2 f_\pm(x)\|_{\mathrm{op}}\leq 2\alpha+\alpha=3\alpha,
\]
so $\nabla f_\pm$ is $3\alpha$-Lipschitz on $\mathbb{R}^d$. 
Setting $L_{f,0}:=3\alpha$, property~(c) holds whenever 
$L_f\geq L_{f,0}$.

For property~(d), we use 
$\nabla f_\pm(x)=2\alpha\,x\pm r\,h^{\beta-1}\,
\Phi^{(k+1)}(x_1/h)\,e_1$. With $r\leq r_{\mathrm{sc}}$, 
$h^{\beta-1}\leq 1$, and $\|x\|\leq R_\Theta$ on $\Theta$,
\[
\|\nabla f_\pm(x)\|
\leq 2\alpha\,R_\Theta+r_{\mathrm{sc}}\|\Phi^{(k+1)}\|_\infty=:G_0,
\]
so~(d) holds whenever $G\geq G_0$.

\medskip
\noindent\textit{Verification of properties (e)--(f): bounded gradient and Lipschitz gradient of $g_\pm$ on $\Theta^+$.}
For $x\in\Theta^+$,
\[
\nabla g_\pm(x)=\nabla q(x)
\pm r\,h^{\beta+k-1}\,\Phi'(x_1/h)\,e_1.
\]
Since $\Theta^+$ is compact and $\nabla q$ continuous, 
$M_q:=\sup_{x\in\Theta^+}\|\nabla q(x)\|<\infty$. Since 
$\beta\geq 2$ and $k\geq 0$, $h^{\beta+k-1}\leq 1$, and 
\[
\|\nabla g_\pm(x)\|\leq M_q+r_{\mathrm{sc}}\|\Phi'\|_\infty
=:G_{g,0}.
\]
Property~(e) holds whenever $G_g\geq G_{g,0}$.

For~(f), 
$\nabla^2 g_\pm(x)=\nabla^2 q(x)\pm r\,h^{\beta+k-2}
\Phi''(x_1/h)\,e_1 e_1^\top$. Since $\Theta^+$ is compact and 
$\nabla^2 q$ continuous, $M_q':=\sup_{x\in\Theta^+}
\|\nabla^2 q(x)\|_{\mathrm{op}}<\infty$. Since $\beta\geq 2$ 
and $k\geq 0$,
\[
\|\nabla^2 g_\pm(x)\|_{\mathrm{op}}
\leq M_q'+r_{\mathrm{sc}}\|\Phi''\|_\infty=:\bar{L}_0,
\]
which makes $\nabla g_\pm$ Lipschitz on $\Theta^+$ with constant 
$\bar{L}_0$. Property~(f) holds whenever $\bar{L}\geq\bar{L}_0$.

\medskip
\noindent\textit{Verification of property (g) and separation of the minimizers.}
Since $f_\pm$ is $\alpha$-strongly convex on $\mathbb{R}^d$, each 
$f_\pm$ has a unique global minimizer $x_\pm^\star$, characterized 
by the first-order optimality condition.

The bump term in $f_\pm$ depends only on $x_1$, so 
$\partial_l f_\pm(x)=2\alpha\,x_l$ for $l\geq 2$, which gives 
$x_{\pm,l}^\star=0$ for $l\geq 2$. For the first coordinate, 
$\partial_1 f_\pm(x)=2\alpha\,x_1\pm r\,h^{\beta-1}\,
\Phi^{(k+1)}(x_1/h)$. Setting $u:=x_1/h$ and 
$\lambda:=r\,h^{\beta-2}$, the optimality condition is
\begin{equation}
\label{eq:lb-implicit-eq}
F_\pm(u,\lambda):=2\alpha\,u\pm\lambda\,\Phi^{(k+1)}(u)=0.
\end{equation}
At $(u,\lambda)=(0,0)$ we have $F_\pm(0,0)=0$ and 
$\partial_u F_\pm(0,0)=2\alpha\neq 0$. By the implicit function 
theorem, there exist $\lambda_0,U_0>0$ depending only on 
$\alpha$ and $\Phi$, and $C^\infty$ functions 
$u_\pm^\star:(-\lambda_0,\lambda_0)\to(-U_0,U_0)$ with 
$u_\pm^\star(0)=0$ solving~\eqref{eq:lb-implicit-eq}. Since 
$f_\pm$ is $\alpha$-strongly convex on $\mathbb{R}^d$, the map 
$x_1\mapsto\partial_1 f_\pm(x)$ is strictly increasing and admits 
a unique zero; hence $u_\pm^\star(\lambda)$ is the unique global 
solution of~\eqref{eq:lb-implicit-eq}, and the resulting point 
$x_\pm^\star$ is the global minimizer of $f_\pm$.

Implicit differentiation of~\eqref{eq:lb-implicit-eq} at 
$\lambda=0$ gives
\[
\frac{du_+^\star}{d\lambda}(0)=-\frac{\Phi^{(k+1)}(0)}{2\alpha},
\qquad
\frac{du_-^\star}{d\lambda}(0)=\frac{\Phi^{(k+1)}(0)}{2\alpha},
\]
so the two branches separate linearly at $\lambda=0$. Since 
$u_\pm^\star\in C^2$, after possibly decreasing $\lambda_0$ there 
exists $M_2>0$ such that, for $|\lambda|\leq\lambda_0$,
\[
\Bigl|u_+^\star(\lambda)+\tfrac{\Phi^{(k+1)}(0)}{2\alpha}\,\lambda
\Bigr|\leq M_2\,\lambda^2,
\quad
\Bigl|u_-^\star(\lambda)-\tfrac{\Phi^{(k+1)}(0)}{2\alpha}\,\lambda
\Bigr|\leq M_2\,\lambda^2.
\]
Combining,
\[
|u_+^\star(\lambda)-u_-^\star(\lambda)|
\geq\frac{|\Phi^{(k+1)}(0)|}{\alpha}|\lambda|-2M_2\lambda^2.
\]
Setting 
$\lambda_1:=\min\{\lambda_0,
|\Phi^{(k+1)}(0)|/(4\alpha M_2)\}$, we have, for 
$|\lambda|\leq\lambda_1$,
\begin{equation}
\label{eq:lb-u-sep}
|u_+^\star(\lambda)-u_-^\star(\lambda)|\geq c_0|\lambda|,
\qquad c_0:=\frac{|\Phi^{(k+1)}(0)|}{2\alpha}>0.
\end{equation}
Returning to the original coordinate $x_1=h\,u$,
\begin{equation}
\label{eq:lb-x-sep}
\|x_+^\star-x_-^\star\|
=h\,|u_+^\star(\lambda)-u_-^\star(\lambda)|
\geq c_0\,h\,|\lambda|=c_0\,r\,h^{\beta-1}.
\end{equation}

By further decreasing $\lambda_1$ if needed, we may assume 
$|u_\pm^\star(\lambda)|\leq 2$ on $|\lambda|\leq\lambda_1$. 
Then, for $h\leq h_0$, 
$|x_{\pm,1}^\star|=h|u_\pm^\star(\lambda)|\leq 2h_0$, and 
$x_{\pm,l}^\star=0$ for $l\geq 2$. Hence 
$x_\pm^\star\in[-2h_0,2h_0]^d\subset\mathrm{int}(\Theta)$ 
by~\eqref{eq:lb-Theta-cube}, which verifies property~(g).

\medskip
\noindent\textit{Kullback-Leibler (KL) divergence between the two hypotheses.}
Let $\mathcal{A}$ be an arbitrary sequential algorithm using $N$ queries. Let $P_\pm^N$ denote the joint law of the 
observation-query trajectory $(z_1,y_1,\ldots,z_N,y_N)$ under 
$g_\pm$ and the algorithm $\mathcal{A}$. At step $t$, the query 
$z_t$ is measurable with respect to past data and the internal 
randomization and conditionally on this past, 
$y_t\sim\mathcal{N}(g_\pm(z_t),\sigma^2)$.

By the chain rule for the KL divergence in sequential sampling 
models~\cite[Lemma~15.1]{lattimore2020bandit},
\begin{equation}
\label{eq:lb-chainrule}
\mathrm{KL}(P_+^N\,\|\,P_-^N)
=\sum_{t=1}^N\mathbb{E}_{P_+}\Bigl[
\mathrm{KL}\bigl(\mathcal{N}(g_+(z_t),\sigma^2)\,\big\|\,
\mathcal{N}(g_-(z_t),\sigma^2)\bigr)\Bigr].
\end{equation}
Since 
$\mathrm{KL}(\mathcal{N}(a,\sigma^2)\,\|\,\mathcal{N}(b,\sigma^2))
=(a-b)^2/(2\sigma^2)$ and 
$|g_+(z)-g_-(z)|=2r\,h^{\beta+k}\,|\Phi(z_1/h)|
\leq 2r\,h^{\beta+k}\,\|\Phi\|_\infty$,
\begin{equation}
\label{eq:lb-KL-bound}
\mathrm{KL}(P_+^N\,\|\,P_-^N)
\leq\frac{2N\,r^2\,h^{2(\beta+k)}\,\|\Phi\|_\infty^2}{\sigma^2}.
\end{equation}

\medskip
\noindent\textit{Choice of parameters $h_N$ and $r_N$.}
Set
\begin{equation}
\label{eq:lb-hN-rN}
h_N:=\min\{N^{-1/(2(\beta+k))},\,h_0\},\qquad
r_N:=c_*\,N^{-1/2}\,h_N^{-(\beta+k)},
\end{equation}
where $c_*>0$ will be chosen below. We distinguish two regimes:
\begin{itemize}
\item Regime~A: $N\geq h_0^{-2(\beta+k)}$, in which case 
$h_N=N^{-1/(2(\beta+k))}$ and $r_N=c_*$;
\item Regime~B: $N<h_0^{-2(\beta+k)}$, in which case 
$h_N=h_0$ and $r_N\leq c_*\,h_0^{-(\beta+k)}$ (since $N\geq 1$).
\end{itemize}
We impose four constraints on $c_*$, sufficient to guarantee that 
the construction is valid in both regimes.

\emph{(i) H\"older constraint.} The condition $r_N\leq r_L$ in 
both regimes is implied by
\begin{equation}
\label{eq:lb-c1}
c_*\leq r_L\,h_0^{\beta+k}.
\end{equation}

\emph{(ii) Strong convexity constraint.} Similarly, 
$r_N\leq r_{\mathrm{sc}}$ in both regimes is implied by
\begin{equation}
\label{eq:lb-c2}
c_*\leq r_{\mathrm{sc}}\,h_0^{\beta+k}.
\end{equation}

\emph{(iii) Implicit function constraint.} The condition 
$|\lambda_N|=r_N\,h_N^{\beta-2}\leq\lambda_1$ holds in Regime~A 
since $|\lambda_N|=c_*\,h_N^{\beta-2}\leq c_*$ (as $\beta\geq 2$ 
and $h_N\leq 1$), and in Regime~B since 
$|\lambda_N|\leq c_*\,h_0^{-(\beta+k)+\beta-2}
=c_*\,h_0^{-(k+2)}$. It is therefore sufficient to require
\begin{equation}
\label{eq:lb-c3}
c_*\leq\lambda_1\,h_0^{k+2}.
\end{equation}

\emph{(iv) KL constraint.} Substituting~\eqref{eq:lb-hN-rN} 
into~\eqref{eq:lb-KL-bound}, the $N$- and $h_N$-dependence 
cancel exactly:
\[
\mathrm{KL}(P_+^N\,\|\,P_-^N)
\leq\frac{2N\,r_N^2\,h_N^{2(\beta+k)}\,\|\Phi\|_\infty^2}{\sigma^2}
=\frac{2c_*^2\,\|\Phi\|_\infty^2}{\sigma^2}.
\]
The condition $\mathrm{KL}(P_+^N\|P_-^N)\leq 1/8$ is implied by
\begin{equation}
\label{eq:lb-c4}
c_*\leq\frac{\sigma}{4\,\|\Phi\|_\infty}.
\end{equation}

We finally fix
\begin{equation}
\label{eq:lb-cstar}
c_*:=\min\Bigl\{r_L\,h_0^{\beta+k},\;
r_{\mathrm{sc}}\,h_0^{\beta+k},\;
\lambda_1\,h_0^{k+2},\;\frac{\sigma}{4\,\|\Phi\|_\infty}\Bigr\}>0,
\end{equation}
which simultaneously satisfies the four 
constraints~\eqref{eq:lb-c1}--\eqref{eq:lb-c4}. By Pinsker's 
inequality~\cite[Lemma~2.5]{tsybakov2009nonparametric}, for the total variation distance $\mathrm{TV}(P_+^N,P_-^N)$ we have:
\begin{equation}
\label{eq:lb-TV}
\mathrm{TV}(P_+^N,P_-^N)\leq\sqrt{\tfrac{1}{2}
\mathrm{KL}(P_+^N\,\|\,P_-^N)}\leq\frac{1}{4}.
\end{equation}

\medskip
\noindent \textit{Reduction to two hypotheses.}
Let $\hat{x}_N\in\Theta$ be the estimator output by the 
algorithm $\mathcal{A}$. Set 
$s_N:=\|x_+^\star-x_-^\star\|/2$, which 
by~\eqref{eq:lb-x-sep} satisfies 
$s_N\geq(c_0/2)\,r_N\,h_N^{\beta-1}$. For 
$\omega\in\{+,-\}$, set 
$A_\omega:=\{\|\hat{x}_N-x_\omega^\star\|\geq s_N\}$.

Define the test $\psi$ which equals $+$ if 
$\|\hat{x}_N-x_+^\star\|<s_N$, and $-$ otherwise. If $\psi=+$, 
then by the triangle inequality,
\[
\|\hat{x}_N-x_-^\star\|
\geq\|x_+^\star-x_-^\star\|-\|\hat{x}_N-x_+^\star\|
>2s_N-s_N=s_N,
\]
so $\{\psi=+\}\subset A_-$. By definition, 
$\{\psi=-\}\subset A_+$. Therefore,
\[
\mathbb{P}_+(A_+)+\mathbb{P}_-(A_-)
\geq\mathbb{P}_+(\psi=-)+\mathbb{P}_-(\psi=+).
\]
The testing 
inequality~\cite[Theorem~2.2]{tsybakov2009nonparametric} gives 
$\mathbb{P}_+(\psi=-)+\mathbb{P}_-(\psi=+)
\geq 1-\mathrm{TV}(P_+^N,P_-^N)\geq 3/4$, 
where we used~\eqref{eq:lb-TV}. Hence 
$\max_{\omega}\mathbb{P}_\omega(A_\omega)\geq 3/8$. By Markov's 
inequality,
\begin{equation}
\label{eq:lb-L2}
\max_{\omega\in\{+,-\}}\mathbb{E}_\omega\|\hat{x}_N-x_\omega^\star\|^2
\geq s_N^2\,\max_{\omega}\mathbb{P}_\omega(A_\omega)
\geq\frac{3}{8}s_N^2
=\frac{3}{32}\|x_+^\star-x_-^\star\|^2.
\end{equation}

\medskip
\noindent\textit{Conversion to optimization error and final rate.}
Since $\hat{x}_N\in\Theta$, $x_\omega^\star\in\mathrm{int}(\Theta)$, 
$\nabla f_\omega(x_\omega^\star)=0$, and $f_\omega$ is 
$\alpha$-strongly convex on $\Theta$,
\[
f_\omega(\hat{x}_N)-f_\omega(x_\omega^\star)
\geq\frac{\alpha}{2}\|\hat{x}_N-x_\omega^\star\|^2.
\]
Combining with~\eqref{eq:lb-L2} and~\eqref{eq:lb-x-sep},
\begin{equation}
\label{eq:lb-opt-error}
\max_{\omega}\mathbb{E}_\omega
\bigl[f_\omega(\hat{x}_N)-f_\omega(x_\omega^\star)\bigr]
\geq\frac{3\alpha}{64}\|x_+^\star-x_-^\star\|^2
\geq\frac{3\alpha\,c_0^2}{64}\,r_N^2\,h_N^{2(\beta-1)}.
\end{equation}

It remains to show that $r_N^2\,h_N^{2(\beta-1)}$ is of the order 
$N^{-(\beta-1)/(\beta+k)}$ in both regimes. By the definition 
of $r_N$, 
$r_N^2\,h_N^{2(\beta-1)}=c_*^2\,N^{-1}\,h_N^{-2(k+1)}$.

In Regime~A, $h_N=N^{-1/(2(\beta+k))}$, so 
$h_N^{-2(k+1)}=N^{(k+1)/(\beta+k)}$, and
\[
r_N^2\,h_N^{2(\beta-1)}
=c_*^2\,N^{-1+(k+1)/(\beta+k)}
=c_*^2\,N^{-(\beta-1)/(\beta+k)}.
\]
In Regime~B, $h_N=h_0$, and the inequality $N<h_0^{-2(\beta+k)}$ 
is equivalent to $h_0^{-2(k+1)}\,N^{-(k+1)/(\beta+k)}>1$. Using 
$N^{-1}=N^{-(\beta-1)/(\beta+k)}\cdot N^{-(k+1)/(\beta+k)}$,
\[
r_N^2\,h_N^{2(\beta-1)}
=c_*^2\,h_0^{-2(k+1)}\,N^{-1}
=c_*^2\,N^{-(\beta-1)/(\beta+k)}\,
\bigl(h_0^{-2(k+1)}\,N^{-(k+1)/(\beta+k)}\bigr)
\geq c_*^2\,N^{-(\beta-1)/(\beta+k)}.
\]
Combining the two regimes,
\begin{equation}
\label{eq:lb-rh-final}
r_N^2\,h_N^{2(\beta-1)}\geq c_*^2\,N^{-(\beta-1)/(\beta+k)}.
\end{equation}
Substituting~\eqref{eq:lb-rh-final} into~\eqref{eq:lb-opt-error},
\begin{equation}
\label{eq:lb-final-bound}
\max_{\omega\in\{+,-\}}\mathbb{E}_\omega
\bigl[f_\omega(\hat{x}_N)-f_\omega(x_\omega^\star)\bigr]
\geq C\,N^{-(\beta-1)/(\beta+k)},
\end{equation}
where $C:=3\alpha\,c_0^2\,c_*^2/64>0$ depends on 
$\alpha,\sigma,L,k,\beta,\Theta$ through 
$c_0,c_*,h_0,C_q,C_\Phi$, but not on $N$ or $d$.

The verifications above show that $g_+,g_-\in 
\mathcal{F}'_{\alpha,\beta,k}(L,L_f,G,G_g,\bar{L})$ whenever 
$L>L_0$, $L_f\geq L_{f,0}$, $G\geq G_0$, $G_g\geq G_{g,0}$, 
$\bar{L}\geq\bar{L}_0$. For the arbitrary sequential algorithm 
$\mathcal{A}$ fixed at the beginning of the proof, 
\eqref{eq:lb-final-bound} yields
\[
\sup_{g\in\mathcal{F}'_{\alpha,\beta,k}(L,L_f,G,G_g,\bar{L})}
\mathbb{E}\bigl[f(\hat{x}_N)-f(x^\star)\bigr]
\geq\max_{\omega}\mathbb{E}_\omega
\bigl[f_\omega(\hat{x}_N)-f_\omega(x_\omega^\star)\bigr]
\geq C\,N^{-(\beta-1)/(\beta+k)}.
\]
Taking the infimum over all sequential algorithms completes the 
proof.
\end{proof}

\section{Discussion and conclusion}
We have studied the problem of minimizing the $k$-th order
derivative $f=\partial_j^k g$ of an unknown
function $g$ under noisy zero-order access to $g$, for all
$k\geq 0$. Our main results, Theorems~\ref{thm:upper-bound}
and~\ref{thm:lower-bound-bumps}, identify the optimal rate in $N$
of the minimax optimization risk over the class
$\mathcal{F}'_{\alpha,\beta,k}(L,L_f,G,G_g,\bar{L})$. The upper
bound scales as 
$d^{(2\beta+k-1)/(\beta+k)}\,N^{-(\beta-1)/(\beta+k)}$, while the
matching lower bound is of the order
$C\,N^{-(\beta-1)/(\beta+k)}$ for some constant $C>0$ independent
of $d$. 
The polynomial rate $N^{-(\beta-1)/(\beta+k)}$ reflects the loss
of regularity induced by the indirect access to derivative
information. Each additional order of differentiation effectively
costs one unit of smoothness in the rate exponent. 



Some questions remain to be explored.
The upper bound carries
the dimension factor $d^{(2\beta+k-1)/(\beta+k)}$, whereas the
lower bound is independent of $d$. Whether the exponent
$(2\beta+k-1)/(\beta+k)$ here is sharp is an open question. Standard
Assouad-type constructions are not directly available because
perturbing $g$ along several coordinates simultaneously must
preserve both the H\"older regularity of $g$ and the strong
convexity of $\partial_j^k g$. This interaction makes the usual
reductions delicate.



Several other questions are also of interest, such as relaxing the strong convexity
assumption on $f$ to mere convexity. This changes the geometry of the
problem since the minimizer is no longer unique nor
well-separated and the techniques of both the upper and lower
bound proofs need to be revisited. Studying one-point
oracle models, in which only a single noisy query of $g$
is available per iteration, would introduce additional
bias-variance trade-offs absent from the two-point feedback
setting. It would also be of interest to develop adaptive
procedures that do not require prior knowledge of the smoothness
parameter $\beta$, as well as procedures that exploit
additional structure such as low intrinsic dimensionality to
mitigate the dependence on $d$. Finally, our analysis treats $f$
as a single coordinate $k$-th derivative; minimizing a mixed
partial derivative of order $k$ would require substantially
different gradient estimators since the two-kernel construction
of Section~\ref{sec:prel} is tailored to the case of one coordinate.

On a broader level, our results suggest that gradient-free
optimization of derivatives is well captured by
the tools of nonparametric statistics. The loss of $k$
derivatives induced by indirect observation is reflected by a corresponding shift of the optimal rate exponent.
We hope that this perspective will be useful for related problems at
the interface of nonparametric statistics and gradient-free
optimization.

\vspace{6mm}

{\bf Acknowledgements}

\vspace{2mm}

The work of Sirine Louati and Alexandre B. Tsybakov was supported by Labex Ecodec (ANR-11-LABEX-0047) and by ANR MaLIP (ANR-25-CE40-3228-01).

{  
\bibliographystyle{unsrtnat}
\bibliography{biblio}
  
} 

\end{document}